\documentclass[11pt,leqno]{article}
\usepackage[margin=1in]{geometry} 
\usepackage{amssymb,amsfonts,amsmath,bbm,mathrsfs,stmaryrd,mathtools}
\usepackage{xcolor}
\usepackage{url}

\usepackage{extarrows}

\usepackage[shortlabels]{enumitem}
\usepackage{tensor}

\usepackage{xr}
\usepackage[T1]{fontenc}
\usepackage[utf8]{inputenc}

\usepackage[colorlinks,
linkcolor=black!75!red,
citecolor=blue,
pdftitle={},
pdfproducer={pdfLaTeX},
pdfpagemode=None,
bookmarksopen=true,
bookmarksnumbered=true,
backref=page]{hyperref}

\usepackage{tikz}
\usetikzlibrary{arrows,calc,decorations.pathreplacing,decorations.markings,decorations.shapes,intersections,shapes.geometric,through,fit,shapes.symbols,positioning,decorations.pathmorphing}

\makeatletter
\newlength\zig@L
\newlength\zig@La
\newlength\zig@Lb

\newcommand{\xzigrightarrow}[2][]{%
  \mathrel{%
    \settowidth{\zig@La}{$\scriptstyle #2$}%
    \settowidth{\zig@Lb}{$\scriptstyle #1$}%
    \zig@L=\zig@La\relax
    \ifdim\zig@Lb>\zig@L \zig@L=\zig@Lb\fi
    \advance\zig@L by 2.2em\relax
    \tikz[baseline=-0.65ex]{%
      \draw[->,
      line cap=round,
      decorate,
      decoration={zigzag,segment length=4pt,amplitude=1.1pt}]%
      (0,0) -- (\zig@L,0)
      node[midway,above=2pt] {$\scriptstyle #2$}%
      \if\relax\detokenize{#1}\relax\else
        node[midway,below=2pt] {$\scriptstyle #1$}%
      \fi
      ;
    }%
  }%
}
\makeatother

\makeatletter
\newcommand{\squigjoin}{1mu} 

\def\sqleft@{\sim}                    
\def\sqmid@{\sim\mkern-\squigjoin}    

\def\rightsquigarrowfill@{%
  \arrowfill@{\sqleft@}{\sqmid@}{\mkern-4mu\succ}%
}

\newcommand{\xrightsquigarrow}[2][]{%
  \ext@arrow 0359\rightsquigarrowfill@{#1}{#2}%
}
\makeatother

\makeatletter
\newcommand*\circled[1]{\tikz[baseline=(char.base)]{
    \node[shape=circle, draw, inner sep=0pt, 
    minimum height={\f@size},] (char) {\vphantom{WAH1g}#1};}}
\makeatother

\makeatletter
\DeclareRobustCommand\widecheck[1]{{\mathpalette\@widecheck{#1}}}
\def\@widecheck#1#2{%
  \setbox\z@\hbox{\m@th$#1#2$}%
  \setbox\tw@\hbox{\m@th$#1%
    \widehat{%
      \vrule\@width\z@\@height\ht\z@
      \vrule\@height\z@\@width\wd\z@}$}%
  \dp\tw@-\ht\z@
  \@tempdima\ht\z@ \advance\@tempdima2\ht\tw@ \divide\@tempdima\thr@@
  \setbox\tw@\hbox{%
    \raise\@tempdima\hbox{\scalebox{1}[-1]{\lower\@tempdima\box
        \tw@}}}%
  {\ooalign{\box\tw@ \cr \box\z@}}}
\makeatother

\usepackage{braket}

\usepackage[amsmath,thmmarks,hyperref]{ntheorem}
\usepackage{cleveref}

\newcommand\nthalias[1]{\AddToHook{env/#1/begin}{\crefalias{lemma}{#1}}}

\nthalias{definition}
\nthalias{example}
\nthalias{examples}
\nthalias{remark}
\nthalias{remarks}
\nthalias{convention}
\nthalias{notation}
\nthalias{construction}
\nthalias{sketch}
\nthalias{theoremN}
\nthalias{propositionN}
\nthalias{corollaryN}
\nthalias{lemma}
\nthalias{proposition}
\nthalias{corollary}
\nthalias{theorem}
\nthalias{conjecture}
\nthalias{question}
\nthalias{assumption}

\creflabelformat{enumi}{#2#1#3}

\crefname{section}{Section}{Sections}
\crefformat{section}{#2Section~#1#3} 
\Crefformat{section}{#2Section~#1#3} 

\crefname{subsection}{\S}{\S\S}
\AtBeginDocument{%
  \crefformat{subsection}{#2\S#1#3}%
  \Crefformat{subsection}{#2\S#1#3}%
}

\crefname{subsubsection}{\S}{\S\S}
\AtBeginDocument{%
  \crefformat{subsubsection}{#2\S#1#3}%
  \Crefformat{subsubsection}{#2\S#1#3}%
}

\theoremstyle{plain}

\newtheorem{lemma}{Lemma}[section]
\newtheorem{proposition}[lemma]{Proposition}
\newtheorem{corollary}[lemma]{Corollary}
\newtheorem{theorem}[lemma]{Theorem}

\theoremstyle{plain}
\theoremnumbering{Alph}

\theoremstyle{plain}
\theorembodyfont{\upshape}
\theoremsymbol{\ensuremath{\blacklozenge}}

\newtheorem{definition}[lemma]{Definition}

\newtheorem{remark}[lemma]{Remark}

\crefname{definition}{definition}{definitions}
\crefformat{definition}{#2definition~#1#3} 
\Crefformat{definition}{#2Definition~#1#3} 

\crefname{ex}{example}{examples}
\crefformat{example}{#2example~#1#3} 
\Crefformat{example}{#2Example~#1#3} 

\crefname{exs}{example}{examples}
\crefformat{examples}{#2example~#1#3} 
\Crefformat{examples}{#2Example~#1#3} 

\crefname{remark}{remark}{remarks}
\crefformat{remark}{#2remark~#1#3} 
\Crefformat{remark}{#2Remark~#1#3} 

\crefname{remarks}{remark}{remarks}
\crefformat{remarks}{#2remark~#1#3} 
\Crefformat{remarks}{#2Remark~#1#3} 

\crefname{convention}{convention}{conventions}
\crefformat{convention}{#2convention~#1#3} 
\Crefformat{convention}{#2Convention~#1#3} 

\crefname{notation}{notation}{notations}
\crefformat{notation}{#2notation~#1#3} 
\Crefformat{notation}{#2Notation~#1#3} 

\crefname{table}{table}{tables}
\crefformat{table}{#2table~#1#3} 
\Crefformat{table}{#2Table~#1#3}

\crefname{lemma}{lemma}{lemmas}
\crefformat{lemma}{#2lemma~#1#3} 
\Crefformat{lemma}{#2Lemma~#1#3} 

\crefname{proposition}{proposition}{propositions}
\crefformat{proposition}{#2proposition~#1#3} 
\Crefformat{proposition}{#2Proposition~#1#3} 

\crefname{propositionN}{proposition}{propositions}
\crefformat{propositionN}{#2proposition~#1#3} 
\Crefformat{propositionN}{#2Proposition~#1#3} 

\crefname{corollary}{corollary}{corollaries}
\crefformat{corollary}{#2corollary~#1#3} 
\Crefformat{corollary}{#2Corollary~#1#3} 

\crefname{corollaryN}{corollary}{corollaries}
\crefformat{corollaryN}{#2corollary~#1#3} 
\Crefformat{corollaryN}{#2Corollary~#1#3} 

\crefname{theorem}{theorem}{theorems}
\crefformat{theorem}{#2theorem~#1#3} 
\Crefformat{theorem}{#2Theorem~#1#3} 

\crefname{theoremN}{theorem}{theorems}
\crefformat{theoremN}{#2theorem~#1#3} 
\Crefformat{theoremN}{#2Theorem~#1#3} 

\crefname{enumi}{}{}
\crefformat{enumi}{#2#1#3}
\Crefformat{enumi}{#2#1#3}

\crefname{assumption}{assumption}{Assumptions}
\crefformat{assumption}{#2assumption~#1#3} 
\Crefformat{assumption}{#2Assumption~#1#3} 

\crefname{construction}{construction}{Constructions}
\crefformat{construction}{#2construction~#1#3} 
\Crefformat{construction}{#2Construction~#1#3} 

\crefname{sketch}{sketch}{Sketches}
\crefformat{sketch}{#2sketch~#1#3} 
\Crefformat{sketch}{#2Sketch~#1#3} 

\crefname{question}{question}{Questions}
\crefformat{question}{#2question~#1#3} 
\Crefformat{question}{#2Question~#1#3} 

\crefname{equation}{}{}
\crefformat{equation}{(#2#1#3)} 
\Crefformat{equation}{(#2#1#3)}

\numberwithin{equation}{section}

\theoremstyle{nonumberplain}
\theoremsymbol{\ensuremath{\blacksquare}}

\newtheorem{proof}{Proof}
\newcommand\pf[1]{\newtheorem{#1}{Proof of \Cref{#1}}}

\newcommand\bG{{\mathbb G}}

\newcommand\bP{{\mathbb P}}

\newcommand\bR{{\mathbb R}}

\newcommand\bV{{\mathbb V}}

\newcommand\bZ{{\mathbb Z}}

\newcommand\wt{\widetilde}

\newcommand{\qedhere}{\mbox{}\hfill\ensuremath{\blacksquare}}

\title{Rigid ternary relations in finite-dimensional Hilbert-space Grassmannians}
\author{Alexandru Chirvasitu}

\begin{document}

\date{}

\newcommand{\Addresses}{{
  \bigskip
  \footnotesize

  \textsc{Department of Mathematics, University at Buffalo}
  \par\nopagebreak
  \textsc{Buffalo, NY 14260-2900, USA}  
  \par\nopagebreak
  \textit{E-mail address}: \texttt{achirvas@buffalo.edu}

}}

\maketitle

\begin{abstract}
  For positive integers $1\le r<d<n$ consider subsets $S\subseteq \mathbb{G}(r,V)$ of the $r$-plane Grassmannian of an $n$-dimensional Hilbert space $V$ saturated in the sense that the $r$-plane $\eta''$ belongs to $S$ whenever it is the orthogonal projection of $\eta'\in S$ onto a $d$-plane through $\eta\in S$. Motivated by such closure operators' natural occurrence in projective-geometry and linear preserver problems, we classify said saturated sets as precisely the disjoint unions of Grassmannian spines, with cores standing in a relation of mutual separation that can be made precise (a spine being the set of $r$-planes containing a fixed core $k$-plane $\pi$ for some $0\le k\le r$). This generalizes the author's results describing saturated $r$-plane sets in the tame dimensional regime $2r\le d$, where the disjoint unions in question by necessity collapse to single spines. 
\end{abstract}


\vspace{.5cm}



\section*{Introduction}

Given positive integers $1\le r<d<n$, \cite[Introduction]{MR5050586} defines the ternary relation $\tau_d=\tau_d(\bullet,\bullet\mid\bullet)$ on the $r$-plane Grassmannian $\bG(r,V)$ of a (real or complex) $n$-dimensional Hilbert space $V$ by
\begin{equation*}
  \tau(\eta,\eta'\mid \eta'')
  \iff
  \exists\left(\zeta\in \bG(d,V)\right)
  \bigg(\eta\le \zeta \wedge \eta''=\text{orthogonal projection }P_{\zeta}\eta'\bigg).
\end{equation*}

That paper's main result concerning that setup, \cite[Theorem 1.8]{MR5050586}, is a rigidity statement valid provided $2r\le d$ (and \emph{only} then): the subsets of $\bG(r,V)$ \emph{$\tau_d$-saturated} in the sense that along with $\eta,\eta'$ they contain all $\eta''$, $\tau_d(\eta,\eta'\mid \eta'')$ are precisely the \emph{Grassmannian $r$-spines}
\begin{equation*}
  \pi^{\uparrow}
  =
  \pi^{\uparrow}_{\bG(r,V)}
  :=
  \left\{\wt{\pi}\in \bG(r,V)\ :\ \wt{\pi}\ge \pi\right\}
\end{equation*}
for $(\le r)$-dimensional $\pi$. 

What originally motivated considering $\tau_d$ was the emergence of $\tau_2$-saturation as instrumental in results on and around Grassmannian geometry in Hilbert spaces: one approach to \cite[Theorem 0.1]{2601.11455v2}, a variant of the celebrated \emph{fundamental theorem of projective geometry} \cite[Theorem 2.3]{pank_wign}, employs precisely that saturation property. The former classifies maps $\bG(V)\to \bG(W)$ between finite-dimensional Hilbert-space Grassmannians respecting lattice operations on \emph{commeasurable} pairs of vector spaces (a term familiar from quantum-mechanics-oriented literature \cite[p.140]{MR2149209}: the orthogonal projections of the spaces in question commute).

In turn, projective-geometry classification results of this nature readily plug in as auxiliary in linear-algebra \emph{preserver problems} \cite{MR2736150,zbMATH01100760,zbMATH07673413,MR4830482}: \cite[Theorem 1.1]{zbMATH05302134} and \cite[Theorem 0.2]{2501.06840v2}, for instance, classifying continuous, spectrum- and commutativity-preserving maps between various types of matrix spaces, both employ the fundamental theorem (in qualitatively different ways). The present considerations are thus somewhat deeper-rooted than a first perusal might suggest. 

\cite[Lemma 1.9]{MR5050586} shows that spines cannot account for all $\tau_d$-saturation in the size regime $2r>d$: two-element $\tau_d$-saturated sets can always be produced in that case. The aim of the present note is to show that that loss of structure brought about by the $2r>d$ ``phase transition'' is only apparent: $\tau_d$-saturated sets can still be described uniformly for all choices of numerical parameters in terms of spines, with the classification specializing to its simpler $2r\le d$ version for reasons that will become transparent once the statement is in place. An auxiliary notion:

\begin{definition}\label{def:distal}
  Let $0\le k,k'<d<n\in \bZ_{\ge 2}$ and $V$ a real or complex $n$-dimensional Hilbert space. Two subspaces $\pi\in \bG(k,V)$ and $\pi'\in \bG(k',V)$ are \emph{(mutually) $d$-distal} (notation: $\pi\oslash_d \pi'$) if $\dim \pi\cap \pi'^{\perp}>d-k'$.
\end{definition}

\begin{theorem}\label{th:spine.union}
  Let $1\le r<d<n\in \bZ_{\ge 3}$ and $V$ an $n$-dimensional real or complex Hilbert space. The $\tau_d$-saturated subsets of $\bG(r,V)$ are precisely the disjoint spine unions
  \begin{equation*}
    \bigsqcup_{i\in I}\left(\pi_i\right)^{\uparrow}_{\bG(r,V)}   
    ,\quad
    \begin{gathered}
      \left\{\pi_i\right\}_{i\in I}\subseteq \bG(\le r,V):=\bigsqcup_{0\le r'\le r}\bG(r',V)\\      
      \forall(i\ne j\in I)
      \left(\pi_i\oslash_d \pi_j\right).
    \end{gathered}    
  \end{equation*}
\end{theorem}
Being mutually $d$-distal entails $k>d-k'\Rightarrow 2r>d$, so in the ``tame'' range $2r\le d$ covered by the earlier work $d$-distal sets of spaces must be singletons; \Cref{th:spine.union}, then, does indeed specialize to the initial cited result. 

\subsection*{Acknowledgments}

I am grateful to \cite{MR5050586}'s anonymous referee for suggesting further examination of the case $2r>d$ left unaddressed by the earlier work. 


\section{Classifying $\tau_d$-saturated sets of $r$-planes}\label{se:clsf}

We refer freely to the \emph{Zariski topology} \cite[\S 2.1]{ms_nonl} on Grassmannians and subspaces thereof (taking for granted, say, the automatic density of non-empty open subsets of $\bG(\bullet,\bullet)$), and decorate the relevant notions with a ``Z'' subscript for clarity (dense$_Z$, open$_Z$, etc.). The caveat, throughout, is that the Zariski topologies involved are those resulting from Grassmannians' \emph{real} algebraic-variety structures: we need orthogonality (hence also complex conjugation) to be continuous. On complex Grassmannians, then, this would be the \emph{Weil restriction} \cite[\S 7.6]{blr_neron} to $\bR$ of the usual complex Zariski topology. 

\begin{remark}\label{re:not.distal.zopen}
  Note incidentally that the $d$-distal relation is closed$_Z$, for large-dimension intersection conditions are so; its negation, then, is open$_Z$: as soon as $\pi\oslash_d \pi'$ fails, it does for $\pi,\pi'$ ranging over respective Zariski neighborhoods of the originals. 
\end{remark}

An elementary linear-algebra exercise provides a few alternative characterizations of the relation of being $d$-distal. 

\begin{lemma}\label{le:d.distal}
  Let $0\le k,k'<d<n\in \bZ_{\ge 2}$ and $V$ an $n$-dimensional real or complex Hilbert space. The following conditions on subspaces $\pi\in \bG(k,V)$ and $\pi'\in \bG(k',V)$ are equivalent.
  \begin{enumerate}[(a),wide]
  \item\label{item:le:d.distal:pi.pi'} $\pi\oslash_d \pi'$.
  \item\label{item:le:d.distal:sml.sm} $\dim \left(\pi+\pi'^{\perp}\right)<n-d+\max(k,k')$.
  \item\label{item:le:d.distal:no.inj} No orthogonal projection $P_{\eta}$, $\pi'\le \eta\in \bG(d,V)$ injects on $\pi$. 
  \item\label{item:le:d.distal:eta.pi'} For all $\eta\in\pi^{\uparrow}_{\bG(<d,V)}$, $\eta\oslash_d \pi'$.
  \item\label{item:le:d.distal:pi.eta'} For all $\eta'\in\pi'^{\uparrow}_{\bG(<d,V)}$, $\pi\oslash_d \eta'$.
  \item\label{item:le:d.distal:eta.eta'} For all $\eta\in\pi^{\uparrow}_{\bG(<d,V)}$ and $\eta'\in\pi'^{\uparrow}_{\bG(<d,V)}$, $\eta\oslash_d \eta'$.  \qedhere
  \end{enumerate}  
\end{lemma}

Note also that while $\oslash_d$ is not (generally) symmetric, one chiral variant does imply the other.

\begin{lemma}\label{le:oslash.achiral}
  For $0\le k'\le k<d<n$ and an $n$-dimensional Hilbert space $V$ we have
  \begin{equation*}
    \forall\left(\pi\in \bG(k,V),\ \pi'\in \bG(k',V)\right)
    \left(
      \pi\oslash_d \pi'
      \xRightarrow{\quad}
      \pi'\oslash_d \pi
    \right).
  \end{equation*}
\end{lemma}
\begin{proof}
  Given the hypothesis, \Cref{le:d.distal}\Cref{item:le:d.distal:sml.sm} reads $\dim \left(\pi+\pi'^{\perp}\right)<n-d+\max(k,k')=n-d+k$. Taking orthogonal complements, $\dim \pi'\cap \pi^{\perp}>d-k$: precisely the condition $\pi'\oslash_d \pi$.
\end{proof}

A slightly more pedantic rendition of \Cref{le:d.distal}\Cref{item:le:d.distal:pi.eta'} would be
\begin{equation*}
  \forall\left(k'\le r<d\right)
  \forall\left(\eta'\in \pi'^{\uparrow}_{\bG(r,V)}\right)
  \left(\pi\oslash_d \eta'\right). 
\end{equation*}

its equivalence to $\pi\oslash_d \pi'$ can be amplified in at least two ways: the first quantifier can be \emph{existential}, and the second can range over dense$_Z$ sets only rather than be fully universal.

\begin{proposition}\label{pr:all.r.planes.k'.oslash}
  Let $\pi\in \bG(k,V)$ and $\pi'\in \bG(k',V)$ for $0\le k,k'<d<n\in \bZ_{\ge 2}$ and an $n$-dimensional Hilbert space $V$.

  The $d$-distal relation $\pi\oslash_d \pi'$ is equivalent to
  \begin{equation*}
    \exists\left(k'\le r<d\right)
    \forall\left(\eta'\in U\overset{\text{dense$_Z$}}{\subseteq}\pi'^{\uparrow}_{\bG(r,V)}\right)
    \left(\pi\oslash_d \eta'\right). 
  \end{equation*}
\end{proposition}
\begin{proof}
  The forward implication is immediate from \Cref{le:d.distal}, given that the displayed condition is formally weaker than \Cref{le:d.distal}\Cref{item:le:d.distal:eta.pi'}; we thus focus on $(\Leftarrow)$, having fixed an $r\in [k',d)$ as in the statement.

  The assumption is that $\pi$ intersects every $(n-r)$-space in a dense$_Z$ subset of $\bG(n-r,\pi'^{\perp})$ along a $(>d-r)$-dimensional subspace. The intersections $\zeta\cap \pi=\zeta\cap\left(\pi\cap \pi'^{\perp}\right)$, though, are generically small: for an open$_Z$ dense$_Z$ $U'\subseteq \bG(n-r,\pi'^{\perp})$ all such with $\zeta\in U'$ will have minimal theoretical dimension $\max(0,\dim \pi\cap \pi'^{\perp}+k'-r)$. It must thus be the case that
  \begin{equation*}
    \dim \pi\cap \pi'^{\perp}+k'-r
    >d-r
    \xRightarrow{\quad}
    \dim \pi\cap \pi'^{\perp}
    >d-k',
  \end{equation*}
  concluding.
\end{proof}

The observation just made has its mirror counterpart. 

\begin{proposition}\label{pr:all.r.planes.k.oslash}
  Let $\pi\in \bG(k,V)$ and $\pi'\in \bG(k',V)$ for $0\le k, k'<d<n\in \bZ_{\ge 2}$ and an $n$-dimensional Hilbert space $V$.

  The $d$-distal relation $\pi\oslash_d \pi'$ is equivalent to
  \begin{equation*}
    \exists\left(k,k'\le r<d\right)
    \forall\left(\eta\in U\overset{\text{dense$_Z$}}{\subseteq}\pi^{\uparrow}_{\bG(r,V)}\right)
    \left(\eta\oslash_d \pi'\right). 
  \end{equation*}
\end{proposition}
\begin{proof}  
  The implication $(\Rightarrow)$ is even simpler in this case: if $\pi$ intersects $\pi'^{\perp}$ along a $(>d-k')$-dimensional subspace, so will the larger $\eta\ge \pi$; the substance is thus again $(\Leftarrow)$.
  
  \begin{enumerate}[(I),wide]
  \item\textbf{: $k\le k'$.} Since $r\ge k'$, $\eta\oslash_d \pi'$ respectively entail $\pi'\oslash_d\eta$ by \Cref{le:oslash.achiral}. This in turn implies $\pi'\oslash_d\pi$ by \Cref{pr:all.r.planes.k'.oslash}, hence $\pi\oslash_d\pi'$ by \Cref{le:oslash.achiral} again, given the assumption $k\le k'$.

  \item\textbf{: $k>k'$.} Suppose $\dim \pi\cap \pi'^{\perp}=d-k'-t$ for some $t\in \bZ_{\ge 0}$. We then have $\dim \left(\pi+\pi'^{\perp}\right)^{\perp}=d-k-t$ (for $k=\max(k,k')$; cf. \Cref{le:d.distal}\Cref{item:le:d.distal:sml.sm}); that value dominating $r-k-t$, an open$_Z$ dense$_Z$ set of $r$-planes containing $\pi$ will intersect $\pi'^{\perp}$ along spaces of dimension $\le \left(d-k'-t\right)+t=d-k$. This contradicts the hypothesis, yielding the conclusion. 
  \end{enumerate}
\end{proof}

Jointly, \Cref{pr:all.r.planes.k'.oslash,pr:all.r.planes.k.oslash} imply

\begin{corollary}\label{cor:all.rs.planes.oslash}
  Let $\pi\in \bG(k,V)$ and $\pi'\in \bG(k',V)$ for $0\le k, k'<d<n\in \bZ_{\ge 2}$ and an $n$-dimensional Hilbert space $V$.

  The $d$-distal relation $\pi\oslash_d \pi'$ is equivalent to
  \begin{equation*}
    \exists\left(k,k'\le r<d\right)
    \begin{gathered}
      \forall\left(\eta\in U\overset{\text{dense$_Z$}}{\subseteq}\pi^{\uparrow}_{\bG(r,V)}\right)\\
      \forall\left(\eta'\in U'\overset{\text{dense$_Z$}}{\subseteq}\pi'^{\uparrow}_{\bG(r,V)}\right)
    \end{gathered}
    \left(\eta\oslash_d \eta'\right). 
  \end{equation*}
  \qedhere
\end{corollary}

This has a consequence pertinent to \Cref{th:spine.union}, recorded in \Cref{pr:red.when.not.dstl}. With a view to stating it, consider the following notion.

\begin{definition}\label{def:spin.d.clsd}
  Let $1\le r<d<n\in \bZ_{\ge 3}$ and $V$ an $n$-dimensional real or complex Hilbert space.
  
  A subset $S\subseteq \bG(r,V)$ is \emph{spine$_d$-closed} if it contains the entire spine $\left(\eta\cap\eta'\right)^{\uparrow}$ whenever $\eta,\eta'\in S$ are not $d$-distal.
\end{definition}

It will be worth noting that the \emph{core} $\pi$ of a spine $\pi^{\uparrow}_{\bG(r,V)}$ is uniquely attached to that spine; this will be implicit in much of the sequel.  

\begin{proposition}\label{pr:red.when.not.dstl}
  Let $1\le r<d<n\in \bZ_{\ge 3}$ and $V$ an $n$-dimensional real or complex Hilbert space.
  
  If $S\subseteq \bG(r,V)$ is spine$_d$-closed in the sense of \Cref{def:spin.d.clsd} then it is of the form specified in the statement of \Cref{th:spine.union}.
\end{proposition}
\begin{proof}

  $S$ will at any rate be a union (possibly not disjoint, a priori) of inclusion-maximal spines $\pi_i^{\uparrow}=\left(\pi_i\right)_{\bG(r,\bV)}^{\uparrow}$, this being true of any set of $r$-planes whatsoever. It will suffice to argue that the cores $\pi_i$ thereof must, under the hypotheses, be mutually $d$-distal. Consider two of the said maximal spines, say $\pi^{\uparrow}$ and $\pi'^{\uparrow}$; proving their cores $d$-distal (in either direction) amounts, by \Cref{cor:all.rs.planes.oslash}, to showing that
  \begin{equation*}
    \forall\left(\eta\in U\overset{\text{dense$_Z$}}{\subseteq}\pi^{\uparrow}_{\bG(r,V)}\right)
    \forall\left(\eta'\in U'\overset{\text{dense$_Z$}}{\subseteq}\pi'^{\uparrow}_{\bG(r,V)}\right)
    \left(\eta\oslash_d \eta'\right). 
  \end{equation*}
  Were this not the case, $\eta$ and $\eta'$ would fail to be $d$-distal when $\eta$ and $\eta'$ range over open$_Z$ dense$_Z$ subsets of $\pi^{\uparrow}$ and $\pi'^{\uparrow}$ respectively (\Cref{re:not.distal.zopen}). $S$ then contains an open$_Z$ dense$_Z$ subset of $\left(\pi\cap \pi'\right)^{\uparrow}_{\bG(r,V)}$ by \Cref{pr:if.2.spines}; generic pairs of $r$-planes in that set are non-$d$-distal, and spine$_d$ closure applied to such pairs will yield the maximality-contradicting conclusion that $\left(\pi\cap \pi'\right)^{\uparrow}_{\bG(r,V)}\subseteq S$.
\end{proof}

The following general principle is employed in \Cref{pr:red.when.not.dstl} and is also discernible in the subsequent proof of \Cref{th:spine.union}. 

\begin{proposition}\label{pr:if.2.spines}
  Let $0\le k,k'\le r<n\in \bZ_{\ge 1}$, $V$ an $n$-dimensional Hilbert space, and $\pi,\pi'$ elements of $\bG(k,V)$ and $\bG(k',V)$  respectively.

  If $S\subseteq \bG(r,V)$ contains $\left(\eta\cap \eta'\right)^{\uparrow}_{\bG(r,V)}$ for $\eta,\eta'$ ranging over open$_Z$ dense$_Z$ subsets of $\pi^{\uparrow}$ and $\pi'^{\uparrow}$ respectively, then $S$ contains an open$_Z$ dense$_Z$ subset of $\left(\pi\cap \pi'\right)^{\uparrow}$. 
\end{proposition}
\begin{proof}
  A number of harmless simplifying assumptions are available:
  \begin{itemize}[wide]
  \item One can always assume $\pi\cap \pi'=\{0\}$ when convenient, by working only with spaces containing that intersection and identifying all such with subspaces of $\left(\pi\cap \pi'\right)^{\perp}$.

  \item Substituting $\pi'\oplus\ell$ for lines
    \begin{equation*}
      \ell\in \text{orthogonal complement }\pi\ominus \left(\pi\cap \pi'\right)\text{ of $\pi\cap \pi'$ in $\pi$}
    \end{equation*}
    or similarly with the roles of $\pi$ and $\pi'$ interchanged, induction reduces the problem to (trivially intersecting) \emph{lines} $\pi$ and $\pi'$. 
  \end{itemize}
  Fix open$_Z$ dense$_Z$ $U\subseteq \pi^{\uparrow}$ and $U'\subseteq \pi'^{\uparrow}$ witnessing the hypothesis. Should it be the case that $2r>n$, $\eta\cap \eta'$ for $\eta\in U$, $\eta'\in U'$ will contain lines $\ell$ for lines ranging over an open$_Z$ dense$_Z$ subset of the projective space $\bP V$; selecting one such and substituting their respective intersections with a fixed hyperplane supplement $H\ge \pi+\pi'$ of $\ell$ for all spaces in sight, the problem reduces (by induction again) to $2r\le n$. In that case, though, $\eta\cap\eta'$ will be generically trivial and the conclusion is immediate. 
\end{proof}

\pf{th:spine.union}
\begin{th:spine.union}
  \Cref{le:d.distal}'s condition \Cref{item:le:d.distal:no.inj} justifies the claim, made in passing in the statement, that the noted union of spines will indeed be disjoint: certainly, two subspaces of $V$ cannot coincide if one does not inject through an orthogonal projection that fixes the other. 

  \Cref{pr:red.when.not.dstl} reduces the proof to arguing that whenever $S\subseteq \bG(r,V)$ is $\tau_d$-saturated and $\eta,\eta'\in S$ are \emph{not} $d$-distal, $S$ contains the entire spine $\left(\pi:=\eta\cap \eta'\right)^{\uparrow}$. Considering only $V$-subspaces containing $\pi$, substitute for every such the orthogonal complement $\bullet \ominus \pi$ therein. This (and induction on the total numerical size $r+d+n$) will effect the substitutions
  \begin{equation*}
    \bullet
    \xzigrightarrow{\ }
    \bullet\ominus \pi
    ,\quad
    \begin{aligned}
      r&\xzigrightarrow{\ }r-\dim \pi\\
      d&\xzigrightarrow{\ }d-\dim \pi\\
      n&\xzigrightarrow{\ }n-\dim \pi.
    \end{aligned}    
  \end{equation*}
  There is thus no loss in making the simplifying assumption that $\eta,\eta'$ intersect trivially and again, are not $\oslash_d$-related (for all dimension estimates in $\eta^{\perp}$ will be unaffected by the $\bullet\ominus \pi$ substitution). What is more, one may always suppose $2r>d$: the case not already settled by \cite[Theorem 1.8]{MR5050586}. The goal, now, is to argue that $S=\bG(r,V)$.
  
  \begin{enumerate}[(I),wide]
  \item\label{item:th:spine.union:pf.most.exhaust}\textbf{: We have $\zeta=\eta+P_{\zeta}\eta'$ for a dense$_Z$ open$_Z$ set of $\zeta\in \eta^{\uparrow}_{\bG(d,V)}$.} The condition amounts to $\zeta=\eta+\eta'+\zeta^{\perp}$, which will certainly be the case if $\eta+\eta'=V$. Were this not so, choose a line $\ell\le \left(\eta+\eta'\right)^{\perp}$ and proceed by induction, working in $V\ominus \ell$.

  \item\label{item:th:spine.union:pf.most.eta.spines}\textbf{: $S\supseteq \pi^{\uparrow}$ for a dense$_{Z}$ open$_Z$ set of $\pi\in \bG(2r-d,\eta)$.} Claim \Cref{item:th:spine.union:pf.most.exhaust} ensures that
    \begin{equation*}
      \{\eta\cap P_{\zeta}\eta'\ :\ \zeta\in \eta^{\uparrow}_{\bG(d,V)}\}
    \end{equation*}
    contains a dense$_Z$ open$_Z$ set of $\pi\in \bG(2r-d,\eta)$. For each such, apply again the substitution $\bullet\xzigrightarrow{}\bullet\ominus \pi$ to conclude that $S\supseteq \pi^{\uparrow}$ given that $\eta\ominus \pi$ and $P_{\zeta}\eta'\ominus \pi$ belong to $S\ominus \pi$: $r$ and $d$ will have been replaced by $d-r$ and $2d-2r$ respectively, and we are in the numerical regime devoid of distal pairs.

  \item\label{item:th:spine.union:pf.contains.dense.open}\textbf{: $S$ contains an open$_Z$ dense$_Z$ subset of $\bG(r,V)$.} The argument just employed applies to all $\eta''\in \bG(r,V)$, sufficiently Zariski-close to $\eta$, with
    \begin{equation*}
      \dim \eta\cap \eta''=2r-d
      ,\quad
      \left\{\eta\cap \eta''\right\}
      \subseteq
      \bG(2r-d,\eta) 
    \end{equation*}
    ranging over a dense$_Z$ open$_Z$ subset of the latter. Write $\eta\mid \eta''$ to denote the relation $\eta$ and $\eta'$ stand in with respect to each other. For sufficiently large $m$ all length-$m$ the $\zeta_m$ resulting from 
    \begin{equation*}
      \eta=:\eta_0\quad
      \mid\quad
      \eta_1\quad
      \mid \quad
      \cdots\quad
      \mid\quad \eta_m
    \end{equation*}
    range over the desired non-empty neighborhood$_Z$ of $\eta$.

  \item\textbf{: Conclusion.} All of the above applies to $\eta'$ fixed as initially and $\eta_0$ ranging over an open$_Z$ dense$_Z$ neighborhood $U\ni \eta\in \bG(r,V)$ of the $r$-plane chosen initially (for $U\subseteq S$, by \Cref{item:th:spine.union:pf.contains.dense.open}). Applying \Cref{item:th:spine.union:pf.most.exhaust} to the pairs $(\eta_0,\eta')$ (in place of $(\eta,\eta')$) will result in an open$_Z$ dense$_Z$ set $U:=\{\pi\}\subseteq \bG(2r-d,V)$ contained in the (now varying) $\eta_0$ with
    \begin{equation*}
      \forall \pi \left(\pi^{\uparrow}_{\bG(r,V)}\subseteq S\right).
    \end{equation*}
    This suffices to conclude: for all open$_Z$ dense$_Z$ $U'\subseteq \bG(2r-d,V)$ the union of all $\pi^{\uparrow}_{\bG(r,V)}$, $\pi\in U'$ is easily seen to exhaust $\bG(r,V)$.  \qedhere
  \end{enumerate}
\end{th:spine.union}


\addcontentsline{toc}{section}{References}

\def\polhk#1{\setbox0=\hbox{#1}{\ooalign{\hidewidth
  \lower1.5ex\hbox{`}\hidewidth\crcr\unhbox0}}}


\Addresses

\end{document}